\documentclass[letterpaper,11pt]{amsart}

\usepackage{epsfig}
\usepackage{amssymb}
\usepackage{amsfonts}
\usepackage{amsmath}
\usepackage{graphicx}
\usepackage{mathrsfs}
\usepackage{stmaryrd}
\usepackage{amsthm}
\usepackage[all]{xy}
\usepackage[top=1in, bottom=1in, left=1in, right=1in]{geometry}
\usepackage{color}
\usepackage[colorlinks=true]{hyperref}
\hypersetup{urlcolor=blue,citecolor=red,linkcolor=blue}

\newtheorem{theorem}{Theorem}[section]
\newtheorem{lemma}[theorem]{Lemma}
\newtheorem{proposition}{Proposition}[section]
\newtheorem{corollary}{Corollary}[section]

\numberwithin{equation}{section}
\title{Lower Bounds  for the Relative Volume of Poincar\'{e}-Einstein Manifolds}

\author{Fang Wang}
\address{School of Mathematical Sciences, Shanghai Jiao Tong University, 800 Dongchuan Rd, Shanghai, 200240}
\email{fangwang1984@sjtu.edu.cn}
\author{Huihuang Zhou}
\address{School of Mathematical Sciences, Shanghai Jiao Tong University, 800 Dongchuan Rd, Shanghai, 200240}
\email{zhouhuihuang@sjtu.edu.cn}
\date{}
\thanks{Wang and Zhou's research are supported in part by National Natural Science Foundation of China No. 11871331 
and Shanghai Science and Technology Innovation Action Plan No. 20JC1413100.}

\begin{document}
\maketitle

\begin{abstract}
In this paper, we show that for a Poincar\'{e}-Einstein manifold $(X^{n+1},g_+)$ with conformal infinity $(M,[\hat{g}])$ of nonnegative Yamabe type, the fractional Yamabe constants of the boundary provide lower bounds for the relative volume. More explicitly, for any $\gamma\in (0,1)$,
$$
\left(\frac{Y_{2\gamma} (M,[\hat{g}])}{Y_{2\gamma} (\mathbb{S}^n , [g_{\mathbb{S}}])}\right)^{\frac{n}{2\gamma}}
\leq  \frac{V(\Gamma_t(p),g_+)}{V(\Gamma_t(0),g_{\mathbb{H}})} 
\leq  \frac{V( B_t(p),g_+)}{V( B_t(0), g_{\mathbb{H}})}
\leq  1,\quad 0<t<\infty,
$$
where  $B_t(p)$, $\Gamma_t(p)$ are the the geodesic ball and geodesic sphere of radius $t$ in $(X,g_+)$ with center at  $p\in X^{n+1}$; and $B_t(0)$,  $\Gamma_t(0)$ are the the geodesic ball and geodesic sphere in $\mathbb{H}^{n+1}$ with center at $0\in\mathbb{H}^{n+1}$. 
\end{abstract}

%%%%%%% section 1. Introducttion
% Known results for Lower bounds and gap
% All main theorems
% Main Idea of proofs
% Notation used throughout of the paper

%%%%%%% section 2. Hyperbolic space
%%%%%%% section 3. Preliminary: distance function and geodesic normal

%%%%%%% section 5.  \gamma\in (0,1)

\section{Introduction}

%In this paper, we consider the lower bounds the fractional powers of Laplacian  on the conformal infinity of a Poincar\'{e}-Einstein manifold introduced by Zworski and Graham in \cite{GZ}. 
%They also can be defined from other point of view, see \cite{CC}

Suppose $\overline{X}^{n+1}$ is a smooth compact manifold with boundary. Denote $X$ the interior of $\overline{X}$ and $M=\partial X$ the boundary. 
%(By "smooth", we always means $C^{\infty}$.) 
Let $\rho$ be a smooth boundary defining function, i.e., 
$$
0\leq \rho\in C^{\infty}(\overline{X}), \quad \rho>0\ \textrm{in $X$},
\quad \rho=0\ \textrm{on $M$}  
\quad\mathrm{and}\quad d\rho|_{M}\neq 0. 
$$
Let $g_+$ be a \textit{Poincar\'{e}-Einstein metric} on $X$, i.e.,   
$(X, g_+)$ is complete  satisfying
$$
Ric_{g_+}=-ng_+,
$$
and $(X,g_+)$ is conformally compact such that $\rho^2 g_+$ extends smoothly to $\overline{X}$. Denote $\hat{g}=\rho^2 g_+|_{TM}$ and we call $(M, [\hat{g}])$ the \textit{conformal infinity} of $(X, g_+)$. 
Direct calculation shows that all the sectional curvature of $(X, g_+)$ approaches to $-1$ as $\rho \rightarrow 0$. Hence $(X, g_+)$ is \textit{asymptotically hyperbolic}.

The study of Poincar\'{e}-Einstein manifolds has become of strong interest through the AdS/CFT correspondence, relating gravitational theories on $X$ with conformal field theories on $M$. See \cite{An1} and many other works.
In mathematics, a basic problem is that how to understand the interior Einstein metric from the boundary conformal geometry. 
In this paper, we mainly focus on the relation between the interior relative volume and the boundary fractional Yamabe constants. Let $\mathbb{H}^{n+1}$ be the hyperbolic space form and $g_{\mathbb{H}}$ the hyperbolic metric. Then the classical Bishop-Gromov comparison theorem implies that 
\begin{equation*}
t\rightarrow \frac{V( B_t(p),g_+)}{V( B_t(0), g_{\mathbb{H}})}\leq 1
\end{equation*}
is a nonincreasing function of $t$, 
where $B_t(p)$ is the geodesic ball of radius $t$ centered at $p\in X$ w.r.t. metric $g_+$, and $B_t(0)$ is the geodesic ball  of the same radius  in $\mathbb{H}^{n+1}$.  
A natural question is that: what is the limit or lower bound of this function when $t\rightarrow +\infty$?

Our work is motivated by the work of Dutta-Javaheri \cite{DJ} and Li-Qing-Shi \cite{LQS}. 
%%%%
\begin{theorem}[\cite{DJ},\cite{LQS}]\label{thm.LQS}
Assume $(X^{n+1},g_+)(n\geq 3)$ is an asymptotically hyperbolic manifold, which is  $C^3$ conformally compact with conformal infinity  of positive Yamabe type. Let $p\in X^{n+1}$ be a fixed point and $t$ is the distance function from $p$. Assume further that
\begin{equation}
Ric_{g_+}\geq -ng_+, \quad R_{g_+} +n(n+1)=o(e^{-2t}).
\end{equation}
Then 
\begin{equation}\label{eq.LQS}
\left(\frac{Y_2 (M,[\hat{g}])}{Y_2 (\mathbb{S}^n , [g_{\mathbb{S}}])}\right)^{\frac{n}{2}} 
\leq  \frac{V(\Gamma_t(p),g_+)}{V(\Gamma_t(0),g_{\mathbb{H}})} 
\leq  \frac{V( B_t(p),g_+)}{V( B_t(0), g_{\mathbb{H}})}
\leq 1,
\quad \forall\ 0<t<+\infty.
\end{equation}
\end{theorem}
%%%%
As applications,  first the inequality (\ref{eq.LQS}) implies the rigidity theorem, i.e., under the assumption of Theorem \ref{thm.LQS}, if $(M,[\hat{g}])$ is conformally equivalent to the round sphere $(\mathbb{S}^n,g_{\mathbb{S}})$, then $(X,g_+)$ is isometric to the hyperbolic space $\mathbb{H}^{n+1}$. 
(The rigidity theorem was also proved in some earlier works under certain geometric conditions; see \cite{Qi1, ST}.) 
Second (\ref{eq.LQS})  also implies that the sectional curvature of $g_+$ is close to $-1$ globally provided that $Y_2 (M,[\hat{g}])$ is close enough to $Y_2 (\mathbb{S}^n , [g_{\mathbb{S}}])$.  
See \cite{LQS} for the details. 

In this paper, we will improve the lower bound of the relative volume for a Poincar\'{e}-Einstein manifold by using the fractional Yamabe constants of its conformal infinity. 
To define the fractional Yamabe constants of $(M, [\hat{g}])$, we need to consider the  fractional GJMS operators on $(M, \hat{g})$. 
Here we follow the definition given in \cite{GZ} and \cite{CC}. 

First, given any boundary smooth representative $\hat{g}$, there exists a \textit{geodesic normal defining function} $x$ and  an identification of $M\times [0,\epsilon)$ with a collar neighborhood of  the boundary, such that in this neighborhood $g_+$ takes the form
$$
g_+=x^{-2}\left(dx^2+g_x\right),
$$
for $g_x$ a one-parameter family of metrics on $M$ with $g_0=\hat{g}$ and having asymptotic expansion that is even in $x$ at least up to order $n$, by\cite{CDLS}. 

For $\gamma\in (0,\frac{n}{2})$, denote $s=\frac{n}{2}+\gamma$.
Assume $\gamma\notin \mathbb{N}$ and  $\frac{n^2}{4}-\gamma^2\notin \mathrm{Spec}(-\Delta_+)$, where $\Delta_+$ is the Lpalcian-Beltrami of $g_+$. 
%Then by the boundness of resolvent in \cite{MM} and Fredholm theory for metric of finite regularity in \cite{Le}, 
By \cite{MM}, if given any $f\in C^{\infty}(M)$, then there  is a unique solution satisfying the following equation:
$$
-\Delta_{+} u-s(n-s)u=0, \quad x^{s-n}u|_{M}=f.
$$
Moreover, $u$ takes the form
$$
u=x^{n-s}F+x^sG, \quad F|_{M}=f, \quad F, G\in C^{\infty}(\overline{X}).
$$
Then the scattering operator $S(s)$ is defined  by 
$$S(s)f=G|_{M}.$$
Notice that the spectral condition can be satisfied by imposing some geometric conditions. 
For example, by \cite{Le1} if $(M, [\hat{g}])$ is of nonnegative Yamabe type, then $\mathrm{Spec}(-\Delta_+)\geq \frac{n^2}{4}$. 
Then the \textit{fractional GJMS operator} of order $2\gamma$ is defined by the renormalised scattering operator:
\begin{equation}\label{def.P}
P^{\hat{g}}_{2\gamma}=d_{\gamma}S\left(\frac{n}{2}+\gamma \right), 
\quad d_{\gamma}=2^{2\gamma}\frac{\Gamma(\gamma)}{\Gamma(-\gamma)}.
\end{equation}
According to \cite{JS}, 
$P^{\hat{g}}_{2\gamma}=(-\Delta_{\hat{g}})^{\gamma}+ L.O.T. .$ 
In particular, when $\gamma=1$, $P^{\hat{g}}_2$ is the conformal Laplacian; when $\gamma=2$, $P^{\hat{g}}_4$ is the Paneitz operator. 
Actually, for any integer $0<k<[\frac{n}{2}]$, $P^{\hat{g}}_{2k}$ is just the classical GJMS operator \cite{GJMS}.
Moreover, the \textit{fractional Q-curvature} is defined by
$$
Q^{\hat{g}}_{2\gamma}=\frac{2}{n-2\gamma}P^{\hat{g}}_{2\gamma}1.
$$
And the \textit{fractional Yamabe constant}  is defined by 
$$
Y_{2\gamma} (M,[\hat{g}])
=\inf_{f\in C^{\infty}(M), f>0}
\frac{\int_{M} fP^{\hat{g}}_{2\gamma}f \mathrm{dV}_{\hat{g}}}{(\int_{M} f^{\frac{2n}{n-2\gamma}} \mathrm{dV}_{\hat{g}})^{\frac{n-2\gamma}{n}}}
=\inf_{\hat{h}\in [\hat{g}]} \frac{\frac{n-2\gamma}{2}\int_M Q^{\hat{h}}_{2\gamma} \mathrm{dV}_{\hat{h}}}{(V(M, {\hat{h}}))^{\frac{n-2\gamma}{n}}}. 
$$

%If $g_+$ is only $C^{k,\alpha}$ conformally compact, then the same idea works for $0<\gamma<\min\{\frac{k+\alpha}{2},\frac{n}{2}\}$. See Section \ref{sec.app} for the details. As to the spectrum condition, it was showed in \cite{Le1} that if the boundary Yamabe constant $Y_{2\gamma} (M,[\hat{g}])\geq 0$, then $\mathrm{Spec}(-\Delta_+)\geq \frac{n^2}{4}$. In this case $P_{2\gamma}$ forms a continuous family of operators for $0<\gamma<\min\{\frac{k+\alpha}{2},\frac{n}{2}\}$.

A standard example is the ball model of the  hyperbolic space $(\mathbb{H}^{n+1}, g_{\mathbb{H}})$, which has conformal infinity $(\mathbb{S}^n, [{g_{\mathbb{S}}}])$, where $g_{\mathbb{S}}$ is the round metric on $\mathbb{S}^n$.
In this case, $ \mathrm{Spec}(-\Delta_+)=[\frac{n^2}{4},+\infty)$.
For any $\gamma \in (0, \frac{n}{2})$,  $P^{g_{\mathbb{S}}}_{2\gamma}$ can be represented by
\begin{equation*}\label{pgamma}
\begin{aligned} 
&P^{g_{\mathbb{S}}}_{2\gamma}=\frac{\Gamma(B+\frac{1}{2}+\gamma)}{\Gamma(B+\frac{1}{2}-\gamma)}, \quad\mathrm{where}\quad
B=\sqrt{-\Delta_{g_{\mathbb{S}}}+\left(\tfrac{n-1}{2}\right)^2}.
\end{aligned}
\end{equation*}
See \cite{Be1}.  Then the fractional Q-curvature is
$$
Q^{g_{\mathbb{S}}}_{2\gamma}=\frac{2}{n-2\gamma} P^{g_{\mathbb{S}}}_{2 \gamma}(1) =  \frac{2}{n-2\gamma} \frac{\Gamma(\frac{n}{2}+\gamma)}{\Gamma(\frac{n}{2}-\gamma)}; 
$$
and the fractional Yamabe constant is
$$
Y_{2\gamma}(\mathbb{S}^n, [{g_{\mathbb{S}}}]) =2^{\frac{2\gamma}{n}}\pi^{\frac{\gamma(n+1)}{n}} \frac{\Gamma(\frac{n}{2}+\gamma)}{\Gamma(\frac{n}{2}-\gamma)}\left[ \Gamma\left(\frac{n+1}{2}\right)\right]^{-\frac{2\gamma}{n}} =\frac{\Gamma(\frac{n}{2}+\gamma)}{\Gamma(\frac{n}{2}-\gamma)} \left|\mathbb{S}^n\right|^{\frac{2\gamma}{n}}.
$$

%The fractional Yamabe problem 

\vspace{0.2in}
Our main result is the following:
\begin{theorem}\label{thm.main}
Assume $(X^{n+1},g_+)$ is a Poincar\'{e}-Einstein manifold  with conformal infinity of non-negative Yamabe type. 
Let $p\in X^{n+1}$ be a fixed point. Let $t$ denote the distance function from $p$. Then for $\gamma\in (0,1)$, we have
\begin{equation}\label{eq.main}
\left(\frac{Y_{2\gamma} (M,[\hat{g}])}{Y_{2\gamma} (\mathbb{S}^n , [g_{\mathbb{S}}])}\right)^{\frac{n}{2\gamma}} \leq \frac{V(\partial B_{g_+}(p,t))}{V(\partial B_{g_{\mathbb{H}}}(0,t))} \leq \frac{V( B_{g_+}(p,t))}{V( B_{g_{\mathbb{H}}}(0,t))}\leq 1,
\quad \forall\ 0\leq t\leq +\infty.
\end{equation}
\end{theorem}
Notice that, by \cite[Theorem 1.2]{GQ1} and \cite[Lemma 3]{WZ}, if $Y_{2} (M,[\hat{g}])\geq 0$, then $Y_{2\gamma} (M,[\hat{g}])\geq 0$ for all $\gamma\in (0,1)$. 
We say that the lower bound of (\ref{eq.main}) is better then (\ref{eq.LQS}) because of the the following comparison theorem given in  \cite{WZ}.
%%%%
\begin{theorem}[\cite{WZ}]\label{thm.Y2Y1}
Assume $(X^{n+1},g_+)$ is a  Poincar\'{e}-Einstein manifold  with conformal infinity $(M, [\hat{g}])$. Assume $\frac{n^2-1 }{4}\notin Spec(-\Delta_+)$ and $Y_1 (M, [\hat{g}])$ can be achieved by some smooth representative $\hat{g}$. Then
\begin{equation}
\frac{Y_2 (M, [\hat{g}])}{Y_2 (\mathbb{S}, [g_{\mathbb{S}}])}\leq \left(\frac{Y_1 (M, [\hat{g}])}{Y_1 (\mathbb{S}, [g_{\mathbb{S}}])}\right)^2 ,
\end{equation}
and the equality holds if and only if $(X,g_+)$ is isometric to the hyperbolic space $\mathbb{H}^{n+1}$.
\end{theorem}
%%%%

%Then a natural question is that in (\ref{eq.main}), which $\gamma\in (0,1)$ provide the best lower bound? Since Theorem \ref{eq.Y2Y1} 

As an application, Theorem \ref{thm.main} implies the following estimate of the sectional curvature by a  similar proof as \cite[Theorem 1.6]{LQS}. 
\begin{corollary}\label{thm.gap}
For any $\gamma\in (0,1)$ and $\varepsilon>0$, there exists $\delta=\delta(\gamma,\epsilon)>0$, such that for any  Poincar\'{e}-Einstein manifold  $(X^{n+1},g_+)$,  if
$
Y_{2\gamma} (M,[\hat{g}])\geq (1-\delta)Y_{2\gamma}(\mathbb{S}^{n},[g_{\mathbb{S}}]), 
$
then
\begin{equation}\label{eq.gap}
|K[g_+]+1|\leq \varepsilon,
\end{equation}
for all sectional curvature $K$ of $g_+$.

\end{corollary}

In this paper, unlike  \cite{DJ} and \cite{LQS}, we have a strong regularity assumption on the conformal compactification of $g_+$  i.e. $\rho^2g_+$ extends smoothly to the boundary. This is mainly because the resolvent analysis  and the scattering theory for $\Delta_+$ require a full Taylor expansion of the metric $\rho^2g_+$ near the boundary, which are used to define the fractional GJMS operators; see \cite{MM} and \cite{JS}. 
However, by the boundary regularity theorem given in  \cite{CDLS}, when the boundary dimension $n$ is even, there may be obstruction to the global smoothness and  $\rho^2g_+$ may have logarithmic terms in its asymptotic expansion, starting from $\rho^n \log \rho$. For this case, it is nature to consider a corresponding version of \cite{MM} and \cite{JS} with polyhomogenous background metric $g_+$, and after that the discussion of $P^{\hat{g}}_{2\gamma}$ for $\gamma\in (0,\frac{n}{2})$ will be similar as the smooth case . 

% It is expected the results of \cite{MM} and \cite{JS} can be extended to such cases 

The paper is outlined as follows: 
In Section \ref{sec.2}, we introduce  the fractional GJMS operators via metric measure space by Case-Chang and give an equivalent representation of the fractional Yamabe constant by using the compactified metric. 
In Section \ref{sec.3}, we recall some basic properties for the cut locus and the distance function of $g_+$. 
In Section \ref{sec.4}, we give some explicit description of the adapted compactification for the model space $\mathbb{H}^{n+1}$. 
In Section \ref{sec.5}, we consider two conformal compactifications for $g_+$: one is the adapted compactification and the other uses the distance function, which coincide in the model space $\mathbb{H}^{n+1}$; and then we prove Theorem \ref{thm.main}. 
In Section \ref{sec.6}, we give some relation between  the Escobar-Yamabe constants and the relative volume. 

%\textbf{Acknowledgment}: 

%%%%%%%%%%%%%%%%%%%%%%%%%%%%%%%%%%%%%%%%%%%%%%
\vspace{0.2in}
\section{Smooth Metric Measure Space and Adapted Boundary Defining Function}\label{sec.2}
In this section, we recall that the fractional GJMS operators via scattering theory can also be identified as some boundary operators associated to the weighted GJMS operators via smooth metric measure space. See  \cite{Ca,CC} for a full discussion. Here we only consider the case
$$
\gamma\in (0,1). 
$$

Let $(X,g_+)$ be a Poincar\'{e}-Einstein metric with conformal infinity $(M, [\hat{g}])$. 
Assume $\hat{g}$ is a smooth boundary representative. 
Let $x$ be the geodesic normal defining function corresponding to $\hat{g}$. 
Then by \cite{CDLS}, near the boundary
$$
\begin{aligned}
g_+& =x^{-2}\left(dx^2 + \hat{g}+x^2g_2+\cdots  \right),
\end{aligned}
$$
where $g_2=-\hat{A}$ and $\hat{A}$ is the Schouten tensor of $\hat{g}$. 

Consider a general boundary defining function $\bar{\rho}$ (not necessarily smooth)  and set $\bar{g}=\bar{\rho}^2g_+$. 
Denote $\bar{\nabla}, \bar{\Delta}$ the covariant derivative and Beltrami-Laplacian of $\bar{g}$;  $\bar{R}$ the scalar curvature of $\bar{g}$ and $\bar{J}=\frac{1}{2n}\bar{R}. $ 
For any $m=1-2\gamma$, set $\bar{\rho}^m=e^{-\psi}$ and define the weighted conformal Laplacian by
\begin{equation}\label{def.L}
\bar{L}_{2,\psi}^{m}=-\bar{\Delta}_{\psi}+\frac{m+n-1}{2}\bar{J}_{\psi}^m,
\end{equation}
where $\bar{\Delta}_{\psi}$ and $\bar{J}_{\psi}^m$ are the weighted Laplacian and weighted scalar curvature defined as follows: 
\begin{equation}\label{def.J}
\begin{aligned}
	 \bar{\Delta}_{\psi} &=\bar{\Delta}-\bar{\nabla} \psi, 
	 \\
	 \bar{J}_{\psi}^m &=\frac{1}{2(m+n)}\left[\bar{R}-2m\frac{\bar{\Delta} \bar{\rho}}{\bar{\rho}}+m(m-1)\left(\frac{1-|\bar{\nabla} \bar{\rho}|^2}{\bar{\rho}^2}\right)\right]
	 \\
	 &=\bar{J} -\frac{m}{n+1}\left(\bar{J}+\frac{\bar{\Delta} \bar{\rho}}{\bar{\rho}}\right).
\end{aligned}
\end{equation}
Here $\bar{L}_{2,\psi}^{m}$ is conformally covariant in the following sense: if $\tilde{g}=\tilde{\rho}^2g_+=\phi^{-2}\bar{g}$, then
\begin{equation}\label{eq.L2}
\tilde{L}_{2,\psi}^m =\phi^{\frac{m+n+3}{2}}\bar{L}_{2,\psi}^m \phi^{-\frac{m+n-1}{2}}. 
\end{equation}

\begin{lemma}[\cite{CC}]\label{lem.CC3}
Assume $\mathrm{Spec}(-\Delta_+)\geq \frac{n^2}{4}-\gamma^2$. 
Let $\bar{\rho}=v^{\frac{2}{n-2\gamma}}$, where $v$ satisfies
$$
-\Delta_+v-\left(\frac{n^2}{4}-\gamma^2\right)v=0, \quad x^{\gamma-\frac{n}{2}}v|_{M}=1. 
$$
Then $\bar{J}_{\psi}^m=0$. 
Such $\bar{\rho}$ is called an \textit{adapted boundary defining function}. 
\end{lemma}

\begin{lemma}[\cite{CC}]\label{lem.CC1}
 Assume 
$$
\bar{\rho}=x+Ax^{1+2\gamma} +o(x^{1+2\gamma})
$$
for some $A\in C^{\infty}(M)$. Given any $f\in C^{\infty}(M)$, the function $U$ is the solution of the boundary value problem
\begin{equation}
\begin{cases}
\bar{L}_{2,\psi}^mU=0, &\textrm{in $X$},
\\
U=f, &\textrm{on $M$},
\end{cases}
\end{equation}
if and only if the function $u=\bar{\rho}^{n-s}U$ is the solution of the Poisson problem
\begin{equation}
\begin{cases}
-\Delta_+u-s(n-s)u=0, &\textrm{in $X$},
\\
u=x^{n-s}F+x^sG, & F, G\in C^{\infty}(\overline{X}),
\\
F=f, &\textrm{on $M$},
\end{cases}
\end{equation}
for $s=\frac{n}{2}+\gamma$. Moreover, 
\begin{equation}
P_{2\gamma}^{\hat{g}} f-\frac{n-2\gamma}{2}d\gamma A f=\frac{d_{\gamma}}{2\gamma} \lim_{\bar{\rho}\rightarrow 0}	\bar{\rho}^m\frac{\partial U}{\partial \bar{\rho}}.
\end{equation}
\end{lemma}

\begin{lemma}[\cite{CC}]\label{lem.CC2}
Assume 
$$
\bar{\rho}=x+\frac{Q^{\hat{g}}_{2\gamma}}{d_\gamma}x^{1+2\gamma}+O(x^3).
$$
Given any $f\in C^{\infty}(M)$ and $U\in W^{1,2}(\overline{X},\bar{\rho}^m dV_{\bar{g}})$ such that  $U|_M =f$, then
\begin{equation}
\int_{M}fP^{\hat{g}}_{2\gamma}f dV_{\hat{g}}
\leq -\frac{d_\gamma}{2\gamma}\int_{X}\left(|\bar{\nabla} U|^2 +\frac{m+n-1}{2}\bar{J}_{\psi}^m U^2\right)\bar{\rho}^m dV_{\bar{g}}+\frac{n-2\gamma}{2}\int_{M}Q^{\hat{g}}_{2\gamma}f^2 dV_{\hat{g}}. 
\end{equation}
\end{lemma}
Note here for $\gamma\in (0,1)$, $d_{\gamma}<0$; see (\ref{def.P}). 

For any $\bar{g}=\bar{\rho}^2g_+$, $\bar{g}|_{TM}=\hat{g}$, define the fractional Yamabe functional by
\begin{equation}\label{def.I}
I_{2\gamma}(\overline{X},\bar{g};U)
=
\frac{-\frac{d_\gamma}{2\gamma}\int_{X}\left(|\bar{\nabla} U|^2 +\frac{m+n-1}{2}\bar{J}_{\psi}^m U^2\right)\bar{\rho}^m dV_{\bar{g}}+\frac{n-2\gamma}{2}\int_{M}Q^{\hat{g}}_{2\gamma}f^2 dV_{\hat{g}}}{\left(\int_M |f|^{\frac{2n}{n-2\gamma}}dV_{\hat{g}}\right)^{\frac{n-2\gamma}{n}}},
\end{equation}
where $U|_{M}=f$.
 Then Lemma \ref{lem.CC1} and Lemma  \ref{lem.CC2} imply that
 
 %%%%
 \begin{lemma} \label{lem.CC4}
 For $\gamma\in (0,1)$, 
$$
Y_{2\gamma}(M, [\hat{g}])
=\inf_{\tiny\begin{array}{c}U\in W^{1,2}(\overline{X},\bar{\rho}^m dV_{\bar{g}})\\ U|_{M}=f>0\end{array}}I_{2\gamma}(\overline{X},\bar{g};U). 
$$
\end{lemma}

%%%%%%%%%%%%%%%%%%%%%%%%%%%%%%%%%%%%%%%%%%%%%%
\vspace{0.2in}
\section{Distance Function for $(X,g_+)$}\label{sec.3}

Suppose $(X^{n+1}, g_+)$  is a Poincar\'{e}-Einstein manifold with conformal infinity $(M,[\hat{g}])$. Assume $\hat{g}$ is a smooth boundary representative and $x$ is the corresponding geodesic normal defining function. Then $|dx|^2_{x^2g_+}\equiv 1$ in a neighborhood of the boundary. 
Extends $x$ to be a global smooth function, which is positive in the interior.
Define
$$
r=-\log\frac{x}{2}, \quad \Sigma_r=\{r=constant\}. 
$$
Fix any $p\in X$ and let
$$ 
t = dist_{g_+} (\cdot,p), \quad \Gamma_t=\{t=constant\}. 
$$
Here $t$ is smooth everywhere in $X\backslash C_p$, where $C_p$ is the cut locus. 
Notice that
$$C_p=Q_p\cup L_p\cup N_p,$$
where $Q_p$ is the set of conjugate points, $N_p$ is the normal cut locus and  $L_p$ is the rest of non-conjugate cut locus. We recall some basic properties for $C_p$ and the functions $r$ and $t$ here.  
%Let $$u=r-t.$$
%%%%
\begin{lemma}[\cite{Oz, IT}]\label{lem.OI}
	For $(X^{n+1},g_+)$ is a smooth complete Riemannian manifold and $p\in X$,
	\begin{itemize}
	\item $Q_p\cup L_p$ is a closed set of Hausdorff dimension no more than $n-1$; 
	\item $N_p$ consists of possibly countably many disjoint smooth hypersurfaces in $X$;
	\item for each $q\in N_p$, there exists an open neighborhood $\Omega$ of $q$ such that $\Omega \cap C_p=\Omega\cap N_p$ is a piece of smooth hypersurface in $X$.
	\end{itemize}
\end{lemma}
%%%%

\begin{lemma}[\cite{LQS}]\label{lem.LQS}
For each $q\in N_p\cap \Gamma_t$, 
\begin{itemize}
\item the outward angle of the corner  on $\Gamma_t$ is less then $\pi$;
\item if $\Omega$ is a neighbourhood of $q$ such that $\Omega\cap C_p=\Omega\cap N_p$, then let $\nu^{\pm}$ be the two outward normal directions to $N_p$ from the interior of $U\backslash N_p$, then
$$
\partial_{\nu^+}t+\partial_{\nu^-} t\geq 0.
$$
\end{itemize}

\end{lemma}

%%%%
\begin{lemma}[\cite{DJ, LQS}] \label{lem.u}
There exists a constant $C>0$ such that the following holds.
\begin{itemize}
\item 
Let $u=r-t$. Then
$$
|u|\leq C, \quad |du|_{g_+}\leq Ce^{-2r}.
$$

\item
At any point where $t$ is smooth,
$$
\left|\langle\nabla r,\nabla t\rangle_{g_+}-1\right| \leq Ce^{-2r}.
$$

\item
Let $X$ and $Y$ be unit vectors orthogonal to $\nabla r$ in $(X,g_+)$. Then
$$
\left|\nabla^2 r(X,Y)-\langle X,Y\rangle_{g_+}\right| \leq Ce^{-2r}. 
$$

\item 
For $t$ sufficiently large, the geodesic sphere 	$\Gamma_t$ is a Lipschitz graph over $M$. 
\end{itemize}
\end{lemma}
%%%%

%%%%
\begin{lemma}[\cite{SY}]\label{lem.SY}
At any point where $t$ is smooth, the Laplacian comparison theorem gives that
$$
\Delta_{+} t \leq  n\coth (t).
$$
\end{lemma}
%%%%

%%%%%%%%%%%%%%%%%%%%%%%%%%%%%%%%%%%%%%%%%%%%%%

\vspace{0.2in}
\section{Model Space $\mathbb{H}^{n+1}$}\label{sec.4}

\subsection{Ball Model}
The hyperbolic space in ball model is given by
$$
(\mathbb{H}^{n+1}, g_{\mathbb{H}} )=
\left(\mathbb{B}^{n+1}, \frac{4dz^2}{(1-|z|^2 )^2}\right)
$$
with conformal infinity $(\mathbb{S}^n , [g_{\mathbb{S}}])$. 
Take the round metric $g_{\mathbb{S}}$ as the boundary representative.
Then the corresponding geodesic normal defining function is
$$
x=\frac{2(1-|z|)}{1+|z|}, 
$$ 
and away from the center
$$
g_{\mathbb{H}} = x^{-2}\left[dx^2 + \left(1-\frac{x^2}{4}\right)^2g_{\mathbb{S}}\right].
$$
%Here $x$ is not smooth at the center $0$ and we can replace it by a smooth one in an arbitrarily  small neigborhood of $0$ if necessary. 
The distance function $t=dist_{g_{\mathbb{H}}}(\cdot, 0)$ is given by
\[ t=\log\frac{1+|z|}{1-|z|}.\]
It follows that
\[ r=-\log\frac{x}{2}=t.\]

\subsection{Adapted boundary defining function}
Notice that $\mathrm{Spec}(-\Delta_{\mathbb{H}})\geq \frac{n^2}{4}$. 
Take $\gamma\in (0,1)$.
Let $v$ be the unique solution satisfies the following equation:
\begin{equation*}
-\Delta_{\mathbb{H}} v -\left(\frac{n^2}{4}-\gamma^2\right) v=0, \quad x^{\gamma-\frac{n}{2}}v|_{\mathbb{S}^n}=1.
\end{equation*}
Moreover $v>0$ and takes the form
\begin{equation}\label{eq.phi}
v=x^{\frac{n}{2}-\gamma}\left(1+O(x^2)\right) 
+ x^{\frac{n}{2}+\gamma}\left(\frac{n-2\gamma}{2d_\gamma}Q^{g_{\mathbb{S}}}_{2\gamma}+O(x^2)\right), \quad \textrm{as}\ x\rightarrow 0.
\end{equation}
By the round symmetry of equation  and uniqueness of solution, there is a one variable function $\Phi(t)$ such that 
\begin{equation}\label{eq.Phi}
	v=\Phi(t). 
\end{equation}
Therefore
$$
\Phi(t)=(2e^{-t})^{\frac{n}{2}-\gamma}\left(1+\frac{n-2\gamma}{2d_\gamma}Q_{2\gamma}^{g_{\mathbb{S}}}(2e^{-t})^{2\gamma}+O(e^{-2t})\right), \quad \textrm{as}\ t\rightarrow +\infty. 
$$

%%%
\begin{lemma}\label{lem.Phi}
The function $\Phi$ satisfies
$$
\frac{d\Phi}{dt}<0, \quad \forall\ t\in (0,+\infty). 
$$
\end{lemma}
\begin{proof}
First the Laplace-Beltrami operator of  $g_{\mathbb{H}}$ is given by
$$
\Delta_{\mathbb{H}}= x^{n+1}\left(1-\frac{x^2}{4}\right)^{-n}\partial_x \left[x^{-n+1}\left(1-\frac{x^2}{4}\right)^{n}\partial_x \right]+x^{2}\left(1-\frac{x^2}{4}\right)^{-2}\Delta_{\mathbb{S}}, 
$$
where $\Delta_{\mathbb{S}}$ is the Laplace-Beltrami operator of round metric $g_{\mathbb{S}}$ on $\mathbb{S}^n$. 
Since  $\Delta_{\mathbb{S}}v=0$, we have
\begin{equation}
x^{n+1}\left(1-\frac{x^2}{4}\right)^{-n}\partial_x \left[x^{-n+1}\left(1-\frac{x^2}{4}\right)^{n}\partial_x v\right]+\left(\frac{n^2}{4}-\gamma^2\right)v=0.
\end{equation}
Let
$$
h(x):=x^{-n+1}\left(1-\frac{x^2}{4}\right)^{n}\partial_x  v .
$$
Notice that $v>0$ implies that  $\partial_xh<0$  for $x\in(0,2)$.
Since $n\geq 3$ and $v$ is a smooth function of $z$ near the center, it is easy to check that $\lim_{x\rightarrow 2}h(x)=0$. 
Hence $h(x)>0$, i.e., $\partial_x v>0$, for all $x\in(0,2)$. Sicne $x=2e^{-t}$, we finish the proof. 
\end{proof}
%%%

%%%%%%%%%%%%%%%%%%%%%%%%%%%%%%%%%%%%%%%%%%%%%
\vspace{0.2in}
\section{Proof of Theorem \ref{thm.main}}\label{sec.5}
Let $(X,g_+)$ be a Poincar\'{e}-Einstein manifold with conformal infinity $(M, [\hat{g}])$ of nonnegative Yamabe type.  First by \cite{Le1}, $Y_2(M, [\hat{g}])\geq 0$ implies that $\mathrm{Spec}(-\Delta_+)\geq \frac{n^2}{4}$. 
We fix some notations as follows. 
\begin{itemize}
\item 
Let $\hat{g}$ be a smooth boundary representative and $\hat{R}$ its scalar curvature. Set $\hat{J}=\frac{1}{2(n-1)}\hat{R}.$
\item
	Let $x$ be the geodesic normal defining function corresponding to $\hat{g}$ and $r=-\log \frac{x}{2}$. Define 
$$\Sigma_{r_0}=\{r=r_0\}\quad \textrm{and}\quad X_{r_0}=\{r<r_0\},\quad \overline{X}_{r_0}=\{r\leq r_0\}. $$
Then for $r_0$ sufficiently large,  $\Sigma_r$ $(r_0\leq r<+\infty)$ provide a smooth  foliation near the boundary. 

\item 
Fix $p\in X$ and let $t=dist_{g_+}(\cdot, p)$ be the distance function. 
Define
$$\Gamma_{t_0}=\{t=t_0\}\quad \textrm{and}\quad B_{t_0}=\{t< t_0\}, \quad \overline{B}_{t_0}=\{t\leq t_0\}. $$
Then for $t_0$ sufficiently large, $\Gamma_t$ $(t_0 \leq t < +\infty)$ provide a Lipschitz  foliation near the boundary. 

%\item Denote $u=r-t$. 

 \item 
Let  $\gamma\in (0,1)$ be fixed.
% and denote $Q_{2\gamma}^{\hat{g}}=Q_{2\gamma}$ for simplicity. 
\end{itemize}

\subsection{Two Compactifications of $g_+$. }
We introduce two types of compactifications for $g_+$, which coincide for the model case. 
\begin{itemize}
\item (\textbf{Type I})
Let $\bar{g}=4e^{-2\bar{r}}g_+=\bar{\rho}^2g_+$, where $\bar{r}=-\log\frac{\bar{\rho}}{2}$, $\bar{\rho}=v^{\frac{2}{n-2\gamma}}$ and $v$ satsfies the equation:
$$ 
-\Delta_+ v-\left(\frac{n^2  }{4}-\gamma^2\right)v=0,\quad 
x^{\gamma-\frac{n}{2}}v|_{M}=1.
$$
Then by \cite{MM} and \cite{JS}, 
\begin{equation}\label{asy.v}
v=x^{\frac{n}{2}-\gamma}F+x^{\frac{n}{2}+\gamma}G, \quad F,G\in C^{\infty}(\overline{X}), \quad F|_{M}=1. 
\end{equation}
Moreover,  $v$ and  $\bar{\rho}$ have the asymptotic expansions as follows: 
$$
\begin{aligned}
v=&\ 
x^{\frac{n-2\gamma}{2}}
\left[1+x^{2\gamma} \left(\frac{n-2\gamma}{2d_{\gamma}}Q^{\hat{g}}_{2\gamma}\right) 
+x^2 \left(\frac{n-2\gamma}{8(1-\gamma)}\hat{J}\right) +O(x^{3})
\right], 
\\
\bar{\rho}=&\ 
x\left[1+x^{2\gamma}\left(\frac{1}{d_\gamma}Q^{\hat{g}}_{2\gamma}\right) +x^2 \left( \frac{1}{4(1-\gamma)}\hat{J}\right) +O(x^{3})\right], 
\end{aligned}
$$
which implies that 
\begin{equation}\label{asy.r}
	\bar{r} =r-\log\left(1+\frac{2^{2\gamma}}{d_\gamma}Q^{\hat{g}}_{2\gamma}e^{-2\gamma r}+O(e^{-2r})\right).
\end{equation}
Let
$$ 
\bar{g}_t= \bar{g}|_{\Gamma_t},\quad \bar{g}_r =\bar{g}|_{\Sigma_r}.
$$
Denote $\bar{\nabla}$ and $\bar{\Delta}$ the covariant derivative and Laplacian-Beltrami of $\bar{g}$; 
$\bar{J}=\frac{1}{2n}\bar{R}$ and $\bar{J}_{\psi}^m$ the scalar curvature and weighted scalar curvature of $\bar{g}$; 
$\bar{\Pi}_{ij}$ the second fundamental form w.r.t. outward normal on $\Sigma_r$;
and $\bar{H}_r=\mathrm{tr}_{\bar{g}_r}\bar{\Pi}$ the mean curvautre of $\Sigma_r$ in $(\overline{X}, \bar{g})$. 
In particular, 
$$
\Sigma_{\infty}=M, \quad \bar{g}_{\infty}=\hat{g}.
$$

\item (\textbf{Type II})
Let $\tilde{g}=4e^{-2\tilde{t}}g_+=\tilde{\rho}^2g_+$, where $\tilde{t}=-\log\left(\frac{\tilde{\rho}}{2}\right)$, $\tilde{\rho}=\Phi^{\frac{2}{n-2\gamma}}(t)$ and $\Phi(t)$ is the one variable function defined in  (\ref{eq.Phi}).  
Then $\tilde{t}$ has the following asymptotic expansion
\begin{equation}\label{asy.t}
	\tilde{t}=t-\log\left(1+ \frac{2^{2\gamma}}{d_\gamma}Q_{2\gamma}^{g_{\mathbb{S}}}e^{-2\gamma t}+O(e^{-2t})\right).
\end{equation}
Let 
$$ 
\tilde{g}_t= \tilde{g}|_{\Gamma_t},\quad \tilde{g}_r=\tilde{g}|_{\Sigma_r}.
$$
Denote $\tilde{\nabla}$ and $\tilde{\Delta}$ the covariant derivative and Laplacian-Beltrami of $\tilde{g}$; $\tilde{J}=\frac{1}{2n}\tilde{R}$ and  $\tilde{J}_{\psi}^m$ the scalar curvautre and weighted scalar curvautre of $\tilde{g}$; 
$\tilde{\Pi}_{ij}$ the second fundamental form w.r.t. outward normal on $\Sigma_r$; and $\tilde{H}_r=\mathrm{tr}_{\tilde{g}_r}\tilde{\Pi}$ the mean curvature of $\Sigma_r$ in $(\overline{X}, \tilde{g})$. 
\end{itemize}

In our setting, $\tilde{t}$ is not  smooth everywhere, but only depending on $t$. So the analysis at non-smooth points is essentially the same as $t$. However, $\bar{r}$ is not only depending on $r$, but smooth everywhere. So the asymptotic analysis of $\bar{r}$ follows from the global regularity easily.

\subsection{The conformal factor from $\bar{g}$ to $\tilde{g}$}
Let $w=\bar{r}-\tilde{t}$ and $\phi=e^{\frac{n-2\gamma}{2}w}$. Then 
$$
\tilde{g} = \phi^{\frac{4}{n-2\gamma}}\bar{g}. 
$$
We first show that the function $w$  is uniformly bounded in the following sense. 

%%%%
\begin{lemma}\label{lem.w}
There exists a constant $C$ such that $|w|\leq C$ and 
\begin{equation}
|\bar{\nabla}w|^2_{\bar{g}} \leq C\left(e^{(2-4\gamma)r}+1\right), \quad a.e.\ \textrm{in $X$}.
\end{equation}
\end{lemma}
%%%%

\begin{proof}
Firstly notice that $F,G\in C^{\infty}(\overline{X})$ and $F|_{M}=1$ in (\ref{asy.v}) implies that 
$$|\nabla G|^2_{g_+}=O(e^{-2r}), \quad |\nabla F|^2_{g_+}= O(e^{-4r}).$$
Then by  (\ref{asy.r}) (\ref{asy.t}) and Lemma \ref{lem.u}, we have that $|w|\leq C$ and
$$
\begin{aligned}
|\nabla \bar{r}|^{2}_{g_+} &=
1+\frac{2^{2\gamma}4\gamma}{d_\gamma}Q^{\hat{g}}_{2\gamma}e^{-2\gamma r} +O(e^{-2r}+e^{-4\gamma r}),
\\
|\nabla \tilde{t}|^{2}_{g_+} &=
1+\frac{2^{2\gamma}2\gamma}{d_\gamma}Q_{2\gamma}^{g_{\mathbb{S}}}e^{-2\gamma t}  +O(e^{-2t}+e^{-4\gamma t}),\quad a.e., 
\\
\langle\nabla \tilde{t},\nabla \bar{r}\rangle_{g_+} &=
1+\frac{2^{2\gamma}2\gamma}{d_\gamma}Q^{\hat{g}}_{2\gamma}e^{-2\gamma r}+
\frac{2^{2\gamma}2\gamma}{d_\gamma}Q_{2\gamma}^{g_{\mathbb{S}}}e^{-2\gamma t}+O(e^{-2r}+e^{-4\gamma r}), \quad a.e..
\end{aligned}
$$
Thus
$$
\langle \nabla w,\nabla w \rangle_{g_+} =O(e^{-4\gamma r}+e^{-2r})
\quad\Longrightarrow \quad
|\bar{\nabla}w |^2_{\bar{g}}=O(e^{(2-4\gamma)r}+1), \quad a.e. .
$$
\end{proof}

%%%%%%
\begin{lemma}\label{lem.phi}
Let $m=1-2\gamma$. Then
$$
w\in W^{1,2}(\overline{X}, \bar{\rho}^mdV_{\bar{g}}), \quad \phi\in W^{1,2}(\overline{X}, \bar{\rho}^mdV_{\bar{g}}).
$$
Moreover,
\begin{itemize}
\item if $\gamma\in[\frac{1}{2},1)$, 	then $w$ and $\phi$ have Lipschitz extensions to $M$;
\item if $\gamma\in(0,\frac{1}{2})$, then $w$ and $\phi$ have $C^{2\gamma}$ extensions to $M$. 
\end{itemize}
\end{lemma}
%%%%%%

\subsection{Energy Functional}
For $r$ sufficiently large, define the Energy functional on $(\overline{X}_r\backslash C_p, \bar{g})$: 
$$
E(\overline{X}_r,\bar{g};\phi)=-\frac{d_\gamma}{2\gamma}\left[
\int_{X_r\backslash C_p}\left(|\bar{\nabla} \phi|^2 + \frac{m+n-1}{2}\bar{J}_{\psi}^m \phi^2 \right) \bar{\rho}^m dV_{\bar{g}}+
\frac{(n-2\gamma)}{2n} \int_{\Sigma_r \backslash C_p}\bar{H}_r\phi^2 \bar{\rho}^{m}  dV_{\bar{g}_r} 
\right].
$$

%%%%
\begin{lemma}
For $r$ sufficiently large, if $\overline{X}_r\cap C_p=\emptyset$, then $E(\overline{X}_r,\bar{g};\phi)=E(\overline{X}_r,\tilde{g};1)$.
\end{lemma}
%%%%
\begin{proof}
Since $\overline{X}_r\cap C_p=\emptyset$, $t$ and hence $\phi$ are smooth all over $\overline{X}_r$. 
Then the conformal transformation $\tilde{g} =\phi^{\frac{4}{n-2\gamma}}\bar{g}$ implies that 
$$
\tilde{H}_r= \phi^{-\frac{2}{n-2\gamma}}
\left( \bar{H}_r + \frac{2n}{n-2\gamma}\phi^{-1}\frac{\partial \phi}{\partial\nu_r}\right),
$$
where $\nu_r$ is the outward normal direction on $\Sigma_r$. 
Consider the weighted conformal Laplacian $\bar{L}_{\psi}^{m}$ defined in (\ref{def.L}). Then by (\ref{eq.L2}), 
$$
\begin{aligned}
\int_{X_r}\frac{m+n-1}{2}\tilde{J}_{\psi}^m \tilde{\rho}^m dV_{\tilde{g}}
=& \int_{X_r}\left(|\bar{\nabla} \phi|^2 + \frac{m+n-1}{2}\bar{J}_{\psi}^m \phi^2 \right) \bar{\rho}^m dV_{\bar{g}}
-\int_{\Sigma_r}\phi\frac{\partial \phi }{\partial \nu_r}\bar{\rho} ^{m}dV_{\bar{g}_r},
\\
\int_{\Sigma_r}\tilde{H}_r \tilde{\rho} ^{m} dV_{\tilde{g}_r}
=& \int_{\Sigma_r}\bar{H}_r \bar{\rho}^{m} \phi ^{2}dV_{\bar{g}_r}
+\frac{2n}{n-2\gamma}\int_{\Sigma_r}\phi \frac{\partial \phi }{\partial \nu_r}\bar{\rho} ^{m}dV_{\bar{g}_r}.
\end{aligned}
$$
Therefore the energy is  conformally invariant:
\begin{align*}
&\  \int_{X_r}\frac{m+n-1}{2}\tilde{J}_{\psi}^m \tilde{\rho}^m dV_{\tilde{g}}
+ \frac{ (n-2\gamma)}{2n}\int_{\Sigma_r}\tilde{H}_r \tilde{\rho}^{m} dV_{\tilde{g}_r}
\\
=&\ \int_{X_r}\left(|\bar{\nabla} \phi|^2 
+ \frac{m+n-1}{2} \bar{J}_{\psi}^m \phi^2 \right) \bar{\rho}^m dV_{\bar{g}}
+ \frac{ (n-2\gamma)}{2n}\int_{\Sigma_r} \bar{H}_r \bar{\rho}^{m} \phi^{2}dV_{\bar{g}_r}. 
\end{align*}
We finish the proof. 
\end{proof}

%%%%
\begin{lemma}\label{lem.E}
For $r$ sufficiently large,	$E(\overline{X}_r,\bar{g};\phi)\leq E(\overline{X}_r\,\tilde{g};1)$. 
\end{lemma}
%%%%
\begin{proof}
We only need to consider the case  $\overline{X}_r\cap C_p\neq \emptyset$. 
For any $\epsilon>0$, due to the compactness of $\overline{X}_r\cap (Q_p\cup L_p)$, there exists finitely many points $p_i\in \overline{X}_r\cap (Q_p\cup L_p)$ so that 
$
B_r^{\epsilon}=\bigcup_{i} B_{\epsilon}(p_i)
$
covers $\overline{X}_r\cap (Q_p\cup L_p)$. Then
\begin{equation}\label{eq.5.1}
\begin{aligned}
&\ \int_{X_r\backslash (B_r^{\epsilon}\cup N_p)}\left(|\bar{\nabla} \phi|^2 + \frac{m+n-1}{2}\bar{J}_{\psi}^m \phi^2 \right) \bar{\rho}^m dV_{\bar{g}}+
\frac{(n-2\gamma)}{2n} \int_{\Sigma_r \backslash (B_r^{\epsilon}\cup N_p)}\bar{H}\bar{\rho}^{m} \phi^2 dV_{\bar{g}_r} \\
= &\ 
\int_{X_r\backslash (B_r^{\epsilon}\cup N_p)}\left( \frac{m+n-1}{2}\tilde{J}_{\psi}^m  \right) \tilde{\rho}^m dV_{\tilde{g}}+
\frac{(n-2\gamma)}{2n} \int_{\Sigma_r \backslash (B_r^{\epsilon}\cup N_p)}\tilde{H}\tilde{\rho}^{m} dV_{\tilde{g}_r} 
\\
&\ +
\int_{X_r\cap \partial B_r^{\epsilon}} \phi \partial_{\nu}\phi dV_{\bar{g}|_{\partial B_r^{\epsilon}}}
+\int_{X_r\cap (N_p\backslash B_r^{\epsilon})} \phi (\partial_{\nu^+}\phi +\partial_{\nu^-}\phi) dV_{\bar{g}|_{N_p}}).
\end{aligned}
\end{equation}
Here $\nu$ is the outward normal direction to $\partial B_r^{\epsilon}$ pointing from the inside of $X_r\backslash B_r^{\epsilon}$ and $\nu^{\pm}$ are  the two outward normal directions pointing from the inside of $X_r\backslash N_p$. 
By Lemma \ref{lem.OI}, $Q_p\cup L_p$ has Hausdorff dimension $\leq n-1$, which together with Lemma \ref{lem.phi} implies that
$$
\int_{X_r\cap \partial B_r^{\epsilon}} \phi \partial_{\nu}\phi dV_{\bar{g}|_{\partial B_r^{\epsilon}}} \longrightarrow 0, \quad \mathrm{as}\ \epsilon\rightarrow 0. 
$$
At $X_r\cap N_p$, direct computation shows that
$$
\partial_{\nu^+}\phi +\partial_{\nu^-}\phi=
-\frac{n-2\gamma}{2}\phi \left[\partial_{\nu^+}\tilde{t}+\partial_{\nu^-}\tilde{t}\right]
=\frac{\Phi'(t)}{\Phi(t)} \phi \left[\partial_{\nu^+}t+\partial_{\nu^-}t  \right].
$$
Notice that $\Phi(t)>0, \Phi'(t)<0$ by  Lemma \ref{lem.Phi}. 
Then Lemma \ref{lem.LQS} implies that
$$
\int_{X_r\cap (N_p\backslash B_r^{\epsilon})} \phi (\partial_{\nu^+}\phi +\partial_{\nu^-}\phi) dV_{\bar{g}|_{N_p}}\leq 0. 
$$
Here $d_{\gamma}<0$ for $\gamma\in (0,1)$. 
Let $\epsilon\rightarrow 0$ in (\ref{eq.5.1}) and we finish the proof. 
\end{proof}

\vspace{0.1in}
\subsection{Curvature Estimate}
%%%%
\begin{lemma}\label{lem.J}
At any point where $t$ is smooth, the weighted scalar curvature $\tilde{J}_{\psi}^{m}$
satisfies
$$\tilde{J}_{\psi}^{m}\leq 0.$$
\end{lemma}
%%%%
\begin{proof}
According to (\ref{def.J}) or \cite[ Lemma 3.2]{CC} , 
$$
\tilde{J}_{\psi}^{m}=\tilde{J} -\frac{m}{n+1}\left(\tilde{J}+ e^{\tilde{t}}\tilde{\Delta}(e^{-\tilde{t}})\right).
$$
By the conformal transformation $\tilde{g} = 4e^{-2\tilde{t}}g_+$, 
$$
\begin{aligned}
\tilde{J} =&\  \frac{e^{2\tilde{t}}}{4}\left( \Delta_+ \tilde{t} -\frac{(n+1)}{2} -\frac{(n-1)}{2}|\nabla \tilde{t} |_{g_+}^2\right),	
\\
\tilde{\Delta}(e^{-\tilde{t} })
=&\ \frac{e^{2\tilde{t} }}{4}\left(\Delta_{+}(e^{-\tilde{t} }) -(n-1)\langle\nabla \tilde{t} , \nabla (e^{-\tilde{t} } ) \rangle_{g_+}\right)
\\
=&\  \frac{e^{\tilde{t} }}{4}\left(-\Delta_+ \tilde{t}  +n|\nabla \tilde{t} |^{2}_{g_+}\right).
\end{aligned}
$$
Therefore
$$
\begin{aligned}
\tilde{J}_{\psi}^{m}=&\  \frac{e^{2\tilde{t} }}{4}\left(\Delta_+ \tilde{t}  -\frac{(n-2\gamma)}{2}|\nabla \tilde{t} |^{2}_{g_+}-\frac{(n+2\gamma)}{2}\right)
\\
=&\ 
-\frac{e^{2\tilde{t}}}{2(n-2\gamma)\Phi(t)}\left[
\Delta_{+}\Phi(t)+\left(\frac{n^2}{4}-\gamma^2\right)\Phi(t)\right].
\end{aligned}
$$
By Lemma \ref{lem.Phi} and Lemma \ref{lem.SY}, $\Phi'(t)\leq  0$ and $\Delta_{+}t \leq n\coth t=\Delta_{\mathbb{H}}t$. Hence we have
$$ 
\begin{aligned}
\Delta_{+}\Phi+\left(\frac{n^2}{4}-\gamma^2\right)\Phi 
=&\ \Phi'(t)\Delta_{+}t + \Phi''(t) +\left(\frac{n^2}{4}-\gamma^2\right)\Phi(t)
\\
\geq&\ \Phi'(t)\Delta_{\mathbb{H}} t + \Phi''(t) +\left(\frac{n^2}{4}-\gamma^2\right)\Phi(t)
\\
=&\  \Delta_{\mathbb{H}}\Phi(t)+\left(\frac{n^2}{4}-\gamma^2\right)\Phi(t)=0. 
\end{aligned}
$$
Thus it follows that $\tilde{J}_{\psi}^{m}\leq 0$. 
\end{proof}

%%%%
\begin{lemma}\label{lem.H1}
For $r$ sufficiently large,
$$
\bar{H}_r \bar{\rho}^m =-\frac{2n\gamma}{d_\gamma}Q^{\hat{g}}_{2\gamma}+O(e^{(2\gamma-2)r}).
$$
\end{lemma}
%%%%

\begin{proof}
Let $\{e_i\}(1\leq i\leq n)$ be an orthonormal basis for $T_q \Sigma_r$ w.r.t.  $g_+$, with unit normal $\nabla r$.  After  the conformal transformation $\bar{g} = 4e^{-2\bar{r}}g_+$, $\{\bar{e}_i=\frac{e^{\bar{r}}}{2}e_i\}$ forms an orthonormal basis for $T_q \Sigma_r$ w.r.t.  $\bar{g}$, with unit normal  $\frac{e^{\bar{r}}}{2}\nabla r$. 
Then we obtain the second fundamental form
$$
\begin{aligned}
\bar{\Pi}_{ij} &=
 \langle \bar{\nabla}_{\bar{e}_i} (\tfrac{e^{\bar{r}}}{2}\nabla r), \bar{e}_j\rangle_{\bar{g}}
= \langle \bar{\nabla}_{{e}_i} (\tfrac{e^{\bar{r}}}{2}\nabla r), {e}_j\rangle_{g_+}
\\&=
 \langle \nabla_{e_i}(\tfrac{e^{\bar{r}}}{2}\nabla r),e_j\rangle_{g_+}
-\tfrac{e^{\bar{r}}}{2}\langle \nabla\bar{r}, \nabla r\rangle_{g_+}\delta_{ij}
\\&=
\tfrac{e^{\bar{r}}}{2}\left( \nabla ^2 r(e_i,e_j)
-\langle\nabla \bar{r},\nabla r\rangle_{g_+} \delta_{ij}\right).
\end{aligned}
$$
Recall that
$$\nabla ^2 r(e_i,e_j)=\delta_{ij}+O(e^{-2r}),\quad
\bar{r} =r-\log\left(1+\frac{2^{2\gamma}}{d_\gamma}Q^{\hat{g}}_{2\gamma}e^{-2\gamma r}+O(e^{-2r})\right),
$$
which implies the estimates of the mean curvature. 

\end{proof}

%%%%
\begin{lemma}\label{lem.H2}
For almost every $r$ sufficiently large, 
$$
\tilde{H}_r\tilde{\rho}^{m}=-\frac{2n\gamma}{d_\gamma}Q^{g_{\mathbb{S}}}_{2\gamma}+O(e^{(2\gamma-2)r}), \quad a.e.
$$
\end{lemma}
\begin{proof}
Let $\{e_i\}(1\leq i\leq n)$ be an orthonormal basis for $T_q \Sigma_r$ w.r.t.  $g_+$, with unit normal $\nabla r$.
After conformal transformation $\tilde{g} = 4e^{-2\tilde{t}}g_+$, $\{\tilde{e}_i=\frac{e^{\tilde{t}}}{2}e_i\}$ forms an orthonormal basis for $T_q \Sigma_r$ w.r.t.  $\tilde{g}$, with unit normal  $\frac{e^{\tilde{t}}}{2}\nabla r$. 
Then at any point where $t$ is smooth, we have the  second fundamental form
$$
\begin{aligned}
\tilde{\Pi}_{ij} &=
 \langle \tilde{\nabla}_{\tilde{e}_i} (\tfrac{e^{\tilde{t}}}{2}\nabla r), \tilde{e}_j\rangle_{\tilde{g}}
= \langle \tilde{\nabla}_{{e}_i} (\tfrac{e^{\tilde{t}}}{2}\nabla r), {e}_j\rangle_{g_+}
\\&=
 \langle \nabla_{e_i}(\tfrac{e^{\tilde{t}}}{2}\nabla r),e_j\rangle_{g_+}
-\tfrac{e^{\tilde{t}}}{2}\langle \nabla\tilde{t}, \nabla r\rangle_{g_+}\delta_{ij}
\\&=
\tfrac{e^{\tilde{t}}}{2}\left( \nabla ^2 r(e_i,e_j)
-\langle\nabla \tilde{t},\nabla r\rangle_{g_+} \delta_{ij}\right).
\end{aligned}
$$
Recall that
$$
\tilde{t}=t-\log\left(1+\frac{2^{2\gamma}}{d_\gamma}Q^{g_{\mathbb{S}}}_{2\gamma}e^{-2\gamma t}+O(e^{-2t})\right),\quad \langle \nabla t,\nabla r\rangle_{g_+}=1+O(e^{-2r}), \quad a.e.
$$
which implies the estimates of the mean curvature. 
\end{proof}

%%%%%%%%%%%%%%%%%%%%%%%%%%%
%% Need some further check

\subsection{Volume Limit}
Notice that the $t$-level set $\Gamma_t$ is a Lipschitz graph over $M$, via a projection along the $r$-geodesic. 
Moreover, for almost every $t$ large enough, the projection of $\Gamma_t\cap C_p$ is a set of measure zero. %%%%
\begin{lemma}\label{lem.V}
When $t\rightarrow \infty$, 
$$
\lim_{t\rightarrow \infty}V(\Gamma_t, \tilde{g}_t)=V(M, \tilde{g}_{\infty}),
$$
$$
\lim_{t\rightarrow \infty}\frac{V(\Gamma_t, g_+)}{V(\Gamma_t, g_{\mathbb{H}})}
=\frac{V(M,\tilde{g}_{\infty})}{V(\mathbb{S}^n,g_{\mathbb{S}})}. 
$$
\end{lemma}
%%%%
\begin{proof}
The first is by the same calculation as (6.5) in \cite{DJ}. For the second, just notice that (\ref{asy.t}) implies that
$$
\frac{V(\Gamma_t, g_+)}{V(\Gamma_t, g_{\mathbb{H}})}
=\frac{(\frac{e^{\tilde{t}}}{2})^nV(\Gamma_t,\tilde{g})}{\sinh^n(t)V(\mathbb{S}^n, g_{\mathbb{S}})}
=\frac{V(\Gamma_t,\tilde{g})}{V(\mathbb{S}^n, g_{\mathbb{S}})}\left[1+o(1)\right].
$$
Let $t\rightarrow \infty$ and we finish the proof.  
\end{proof}

%%%%%%%%%%%%%%%%%%%%%%%%%%%

\subsection{Proof of Theorem \ref{thm.main}}
Now we consider the Yamabe functional
$$
I_{2\gamma}(\overline{X}, \bar{g};\phi)
=\frac
{-\frac{d_\gamma}{2\gamma}\int_{X}\left(|\nabla \phi|^2 + \frac{m+n-1}{2}\bar{J}_{\psi}^m \phi^2 \right)\bar{\rho}^m dV_{\bar{g}}
+\frac{(n-2\gamma)}{2}\int_{M}Q^{\hat{g}}_{2\gamma} \phi^2 dV_{\hat{g}} }
{\left(\int_{M}|\phi|^{\frac{2n}{n-2\gamma}}dV_{\hat{g}}\right)^{\frac{n-2\gamma}{n}}}. 
$$
Notice that $\bar{J}_{\psi}^m$, $\bar{H}_r$ uniformly bounded since $\bar{g}$ is smooth, and $\hat{g}=\bar{g}_{\infty}$. 
Then by Lemma \ref{lem.E}, \ref{lem.J}, \ref{lem.H1} ,\ref{lem.H2}, together with Lemma \ref{lem.phi}, we have
$$
\begin{aligned}
&\ 
-\frac{d_\gamma}{2\gamma}\int_{X}\left(|\nabla \phi|^2 + \frac{m+n-1}{2}\bar{J}_{\psi}^m \phi^2 \right)\bar{\rho}^m dV_{\bar{g}}
+\frac{(n-2\gamma)}{2}\int_{M}Q^{\hat{g}}_{2\gamma} \phi^2 dV_{\hat{g}} 
\\
\leq &\ 
\liminf_{r\rightarrow \infty}
-\frac{d_\gamma}{2\gamma}\left[
\int_{X_r} \left(|\nabla \phi|^2 + \frac{m+n-1}{2}\bar{J}_{\psi}^m \phi^2 \right) \bar{\rho}^m dV_{\bar{g}}+
\frac{(n-2\gamma)}{2n} \int_{\Sigma_r }\bar{H}\phi^2 \bar{\rho}^{m}  dV_{\bar{g}_r} 
\right]
\\
= &\ 
\liminf_{r\rightarrow \infty}
-\frac{d_\gamma}{2\gamma}\left[
\int_{X_r\backslash C_p} \left(|\nabla \phi|^2 + \frac{m+n-1}{2}\bar{J}_{\psi}^m \phi^2 \right) \bar{\rho}^m dV_{\bar{g}}+
\frac{(n-2\gamma)}{2n} \int_{\Sigma_r \backslash C_p}\bar{H}\phi^2 \bar{\rho}^{m}  dV_{\bar{g}_r} 
\right]
\\
\leq &\ 
\liminf_{r\rightarrow \infty} -\frac{d_\gamma}{2\gamma}\left[\int_{X_r\backslash C_p}\frac{m+n-1}{2}\tilde{J}_{\psi}^m \tilde{\rho}^m dV_{\tilde{g}}
+ \frac{ (n-2\gamma)}{2n}\int_{\Sigma_r\backslash C_p}\tilde{H}_r \tilde{\rho}^{m} dV_{\tilde{g}_r}
\right]
\\
\leq &\ 
\frac{(n-2\gamma)}{2} Q^{g_{\mathbb{S}}}_{2\gamma} V(M,\tilde{g}_{\infty}).
\end{aligned}
$$
This implies that 
$$
I_{2\gamma}(\overline{X}, \bar{g};\phi)
\leq \frac{\frac{(n-2\gamma)}{2} Q^{g_{\mathbb{S}}}_{2\gamma} V(M,\tilde{g}_{\infty})}{\left(\int_{M}|\phi|^{\frac{2n}{n-2\gamma}}dV_{\bar{g}_{\infty}}\right)^{\frac{n-2\gamma}{n}}}
=\frac{(n-2\gamma)}{2} Q^{g_{\mathbb{S}}}_{2\gamma} V(M,\tilde{g}_{\infty})^{\frac{2\gamma}{n}}.
$$
Therefore by Lemma \ref{lem.V} and Bishop-Gromov,
$$
\left(\frac{Y_{2\gamma}(M, [\hat{g}])}{Y_{2\gamma} (\mathbb{S}^n, [g_{\mathbb{S}}])}\right)^{\frac{n}{2\gamma}}
\leq 
\left(\frac{I_{2\gamma}(\overline{X}, \bar{g};\phi)}{Y_{2\gamma} (\mathbb{S}^n, [g_{\mathbb{S}}])}\right)^{\frac{n}{2\gamma}}
\leq \frac{V(M,\tilde{g}_{\infty})}{V(\mathbb{S}^n,g_{\mathbb{S}})}
=\lim_{t\rightarrow \infty}\frac{V(\Gamma_t, g_+)}{V(\Gamma_t, g_{\mathbb{H}})}
\leq  \frac{V(\Gamma_t, g_+)}{V(\Gamma_t, g_{\mathbb{H}})}. 
$$
We finish  the proof of Theorem \ref{thm.main}.

%%%%%%%%%%%%%%%%%%%%%%%%%%%%%%%%%%%%
\vspace{0.2in}
\section{The Escobar-Yamabe Constant}\label{sec.6}
For a general compact Riemannian manifold with boundary $(\overline{Y}^{n+1},\partial Y,\bar{h})$, Escobar \cite{Es1, Es2, Es3} introduced two types of Yamabe constants: 
$$
\begin{aligned}
Y_a(\overline{Y},\partial Y; [\bar{h}])=&\ 
\inf_{0<\varphi \in C^{\infty}(\overline{Y})} \frac{\int_{Y}\left(\frac{4n}{n-1}|\nabla \varphi|_{\bar{h}}^2 + R_{\bar{h}}\varphi^2\right)dV_{\bar{h}}
+2\int_{\partial Y}H_{\bar{h}}\varphi^2 dV_{\hat{h}} }{\left(\int_{Y}\varphi^{\frac{2(n+1)}{n-1}}dV_{\bar{h}}\right)^{(n-1)/(n+1)}};
\\
Y_b(\overline{Y},\partial Y;[\bar{h}])=&\ 
\inf_{0<\varphi \in C^{\infty}(\overline{Y})} \frac{\int_{Y}\left(\frac{4n}{n-1}|\nabla \varphi|_{\bar{h}}^2 + R_{\bar{h}}\varphi^2\right)dV_{\bar{h}}
+2\int_{\partial Y}H_{\bar{h}}\varphi^2 dV_{\hat{h}} }{\left(\int_{\partial Y}\varphi^{\frac{2n}{n-1}}dV_{\hat{h}}\right)^{(n-1)/n}}.
\end{aligned}
$$
where $\hat{h}=\bar{h}|_{\partial Y}$, $R_{\bar{h}}$ is the interior scalar curvautre and $H_{\bar{h}}$ is the boundary mean curvature of $\bar{h}$. 

Now, we follow the notations in Section \ref{sec.5}. 

%%%%%%%%%%%%%%%%%%%%%%%%%%%%%%%%%%%%%%%%%%
\subsection{Relating $Y_b$ and the Volume.}
First, we notice that while $(\overline{X}, M, \bar{g})$ is a conformal compactification of some Poincar\'{e}-Einstein manifold $(X, g_+)$ with conformal infinity $(M, [\hat{g}])$, then
$$
Y_b(\overline{X},M;[\bar{g}])=\frac{4n}{n-1} Y_1(M,[\hat{g}]) . 
$$
The proof of Theorem \ref{thm.main} directly implies the following

\begin{proposition} 
Let $\tilde{g}=4e^{-2\tilde{t}}g_+$ where $\tilde{t}$ is defined in (\ref{asy.t}) with $\gamma=\frac{1}{2}$ and $\tilde{g}_{\infty}=\tilde{g}|_{M}$. Then
$$
\left(\frac{Y_b(\overline{X},M;[\bar{g}])}{Y_b(\overline{\mathbb{S}}_+^{n+1},\mathbb{S}^n;[g_{\mathbb{S}^{n+1}}])}\right)^n
\leq  \frac{V(M,\tilde{g}_{\infty})}{V(\mathbb{S}^n,g_{\mathbb{S}^n})}
\leq \frac{V(\Gamma_t, g_+)}{V(\Gamma_t, g_{\mathbb{H}})}, \quad 0<t <\infty.
$$
\end{proposition}

%%%%%%%%%%%%%%%%%%%%%%%%%%%%%%%%%%%%%%%%%%
\subsection{Relating $Y_a$ with the  Volume.}
Second,  we consider another pair of  conformal compactifications as follows, which are actually Type I  and Type II with the parameter $\gamma=\frac{n}{2}+1$. In model case, they are both the half sphere compactification of $\mathbb{H}^{n+1}$. 
Here we list them separately since the asymptotic expansions are different from $\gamma\in (0,1)$.

\begin{itemize}
\item  Let $\bar{g}=4e^{-2\bar{r}}g_+$, where $\bar{r}=-\log\frac{\bar{\rho}}{2}$, $\bar{\rho}=v^{-1}$ and $v$ satisfies 
$$ 
-\Delta_+ v+(n+1)v=0,\quad 
xv|_{M}=1.
$$
Then $v$ and  $\bar{\rho}$ have the asymptotic expansions as follows: 
$$
\begin{aligned}
v=&\ 
x^{-1}
\left[1+x^2 \left(\frac{1}{2n}\hat{J}\right) +O(x^{3})
\right], 
\\
\bar{\rho}=&\ 
x\left[1-x^2 \left(\frac{1}{2n}\hat{J} \right)+O(x^{3})\right], 
\end{aligned}
$$
which implies that 
\begin{equation}\label{asy.r2}
	\bar{r} =r+\log\left(1+\frac{2}{n}\hat{J} e^{-2 r}+O(e^{-3r})\right).
\end{equation}

\item	 Let $\tilde{g}=4e^{-2\tilde{t}}g_+$, where 
\begin{equation}\label{asy.t2}
\tilde{t}=t+\log(1+e^{-2t}).
\end{equation}

\end{itemize}

Following a similar discussion as in Section \ref{sec.5}, we have that $w=\bar{r}-\tilde{t}$  is uniformly bounded and
$$
\langle\nabla \tilde{t},\nabla \bar{r}\rangle_{g_+}=1+O(e^{-2r}), \quad 
|\nabla w|_{g_+}=O(e^{-2r}). 
$$
Set $\varphi=e^{\frac{n-1}{2}w}$, then $\varphi, w\in W^{1,2}(\overline{X},g)$ implies that $\varphi$ and $w$ have Lipschitz extensions to $M$. 
%%%
\begin{lemma}
At any point where $t$ is smooth, the scalar curvature of $\tilde{g}$ satisfies
$$ 
\tilde{R}\leq n(n+1),
$$
and the second fundamental form of $(\Sigma_r , \tilde{g}_r)$ converges to $0$ as $r \rightarrow \infty$. 
\end{lemma}
\begin{proof}
For the scalar curvature, $\tilde{g}=4e^{-2\tilde{t}}g_+$, we have
$$ 
\tilde{R}=\frac{e^{2\tilde{t}}}{4}\left(-n(n+1)+2n\Delta_{+}\tilde{t}-n(n-1)|\nabla \tilde{t}|^2_{g_+}\right).
$$
Recall $\tilde{t}=t+\log(1+e^{-2t})$ and $\Delta_+ t\leq n\coth (t)$, then
%\[ \nabla \tilde{t}=\tanh(t), \quad \Delta_{+}s=\tanh(t)\Delta_+ t +(1-\tanh^{2}(t))|\nabla t|^{2}_{g_+}.\]
%Thus
$$
\begin{aligned}
\tilde{R}
=&\ \frac{1}{4}e^{2t}\left(1+e^{-2t}\right)^2 \left[-n(n+1)+2n \tanh(t)\Delta_+ t +2n (1-\tanh^{2}(t))-n(n-1) \tanh^{2}(t)\right]
\\
=&\ \cosh^2 (t)\left[2n \tanh(t)\Delta_+ t -n(n-1)-n(n+1) \tanh^{2}(t)\right]
\\
\leq&\ 
 \cosh^2 (t)\left[2n^2 \tanh(t) \coth(t) -n(n-1)-n(n+1) \tanh^{2}(t)\right]\\
 =&\ n(n+1). 
\end{aligned}
$$

For the second fundamental form, let $\frac{e^{\tilde{t}}}{2}\nabla r$
be the unit normal to the level set $(\Sigma_r, \tilde{g}_r)$. Choose an orthonorma basis $\{e_i\}(1\leq i\leq n)$ for $T_{q}\Sigma_r$ w.r.t. metric $g_+$, then $\tilde{e}_i = \frac{e^{\tilde{t}}}{2} e_i$ is an orthonormal basis for $T_q \Sigma_r$ with metric $\tilde{g}_r$. Thus
$$
\begin{aligned}
\tilde{\Pi} (\tilde{e}_i,\tilde{e}_j)
 %=\langle\nabla_{e_i}N,e_j\rangle_{g_+}-\langle\nabla \tilde{t},N\rangle_{g_+}\delta_{ij}
=&\ \frac{e^{\tilde{t}}}{2}\left(\nabla^2 r(e_i,e_j)-\langle\nabla \tilde{t},\nabla r\rangle_{g_+}\delta_{ij}\right)\\
=&\ \frac{1}{2}e^t \left(1+e^{-2t}\right)\left[\delta_{ij}+O(e^{-2r})-\frac{1-e^{-2t}}{1+e^{-2t}}\left(1+O(e^{-2r})\right)\delta_{ij}\right]
=O(e^{-r}).
\end{aligned}
$$
\end{proof}

%%%%%%
\begin{proposition}
Let $\tilde{g}=4e^{-2\tilde{t}}g_+$ where $\tilde{t}$ is defined in (\ref{asy.t2}). Then
\begin{equation}\label{eq.Ya}
\left(\frac{Y_a(\overline{X},M;[\bar{g}])}{Y_a(\overline{\mathbb{S}}_+^{n+1},\mathbb{S}^n;[g_{\mathbb{S}^{n+1}}])}\right)^{\frac{n+1}{2}}
\leq  \frac{V(\overline{X},\tilde{g})}{V(\overline{\mathbb{S}}_+^{n+1}, g_{\mathbb{S}^{n+1}})}.
%\leq  \frac{V(B_t(p), g_+)}{V(B_t(0), g_{\mathbb{H}})}, \quad 0<t<\infty. 
\end{equation}
\end{proposition}
%%%%%%
\begin{proof}
Consider the conformal transformation  $\tilde{g}=\varphi^{\frac{4}{n-1}}\bar{g}$ and the Yamabe fractional
$$
I_a(\overline{X}, \bar{g};\varphi)
=\frac{\int_{X}\left(\frac{4n}{n-1}|\bar{\nabla} \varphi|_{\bar{g}}^2 + \bar{R}\varphi^2\right)dV_{\bar{g}}+2\int_{M}\bar{H}\varphi^2 dV_{\hat{g}} }
{\left(\int_{X}\varphi^{\frac{2(n+1)}{n-1}}dV_{\bar{g}}\right)^{\frac{n-1}{n+1}}}.
$$
Then by a similar proof as Theorem \ref{thm.main},
we have
$$
\begin{aligned}
I_a(\overline{X}, \bar{g};\varphi)
\leq&\  \liminf_{r\to\infty} \frac{\int_{X_r\backslash C_p}
 \left(\frac{4n}{n-1}|\bar{\nabla} \varphi|_{\bar{g}}^2+ \bar{R}\varphi^2\right)dV_{\bar{g}}
+2\int_{\Sigma_r\backslash C_p}\bar{H}\varphi^2 dV_{\bar{g}_r}}
{\left(\int_{X_r}\varphi^{\frac{2(n+1)}{n-1}}dV_{\bar{g}}\right)^{\frac{n-1}{n+1}} }
\\
\leq &\ 
\liminf_{r\to\infty}\frac{\int_{X_r\backslash C_p}\tilde{R}dV_{\tilde{g}} 
+2\int_{\Sigma_r\backslash C_p}\tilde{H}_rdV_{\tilde{g}_r}} 
{\left(\int_{X_r}dV_{\tilde{g}}\right)^{\frac{n-1}{n+1}} } 
\\
\leq &\ n(n+1) V(\overline{X}, \tilde{g})^{\frac{2}{n+1}}.
\end{aligned}
$$
Therefore $Y_a(\overline{X},M;[\bar{g}])\leq\ n(n+1) V(\overline{X}, \tilde{g})^{\frac{2}{n+1}}$.
Since $Y_a(\overline{\mathbb{S}}_+^{n+1},\mathbb{S}^n;[g_{\mathbb{S}^{n+1}}])=n(n+1)V(\overline{\mathbb{S}}_+^{n+1}, g_{\mathbb{S}^{n+1}})^{\frac{2}{n+1}}$, we finish the proof. 
\end{proof}

Notice that, for all $0<t<\infty$, $ \frac{e^{\tilde{t}}}{2}=\cosh(t)$ and $\cosh^2 (t)g_{\mathbb{H}}$ is the standard round metric on $\mathbb{S}^{n+1}_+$. Moreover, the ratio function
$$
\eta(t)=\frac{V(\Gamma_t(p), g_+)}{V(\Gamma_t(0), g_{\mathbb{H}})}
=\frac{(\frac{e^{\tilde{t}}}{2})^n V(\Gamma_t(p), \tilde{g})}{\cosh^n (t)V(\Gamma_t(0), g_{\mathbb{S}^{n+1}})}
=\frac{V(\Gamma_t(p), \tilde{g})}{V(\Gamma_t(0), g_{\mathbb{S}^{n+1}})}
$$
is nonincreasing. Hence
$$
\frac{V(B_t(p), \tilde{g})}{V(B_t(0), g_{\mathbb{S}^{n+1}})}
$$
is nonincreasing with limit 
$$
\frac{V(\overline{X},\tilde{g})}{V(\overline{\mathbb{S}}_+^{n+1}, g_{\mathbb{S}^{n+1}})}.
$$
Therefore (\ref{eq.Ya}) provide a lower bound for some "mean value" of $\eta(t)$.

%%%%%%%%%%%%%%%%%%%%%%%%%%%%%%%%%%%%%%%%%%%%%%%%%%%%%
% Definition of Fractional Operator with background metric of finite regularity

%%%%%%%%%%%%%%%%%%%%%%%%%%%%%%%%%%%%%%%%%%%%%%%%%%%%%%%
%%% Appendix sections. žœÂŒÕÂœÚ, ·Ç±ØÑ¡
%%%%%%%%%%%%%%%%%%%%%%%%%%%%%%%%%%%%%%%%%%%%%%%%%%%%%%%
%\begin{appendix}

%\end{appendix}

\end{document}